\documentclass[11pt,a4paper]{article}
\usepackage{graphicx}
\usepackage{amsmath}
\usepackage{graphicx}
\usepackage{verbatim}
\usepackage{color}
\usepackage{subfigure}
\usepackage{float}
\usepackage{authblk}
\usepackage[colorlinks]{hyperref}
\usepackage{doi}

\usepackage{graphicx,amsmath,amsfonts,latexsym,amssymb,amsthm,amscd}
\usepackage{latexsym,amscd}
\newcommand{\eq}[1]{\begin{equation}\label{#1}}
\newcommand{\en}{\end{equation}}
\newcommand{\eqst}[1]{\begin{equation*}\label{#1}}
\newcommand{\enst}{\end{equation*}}
\newcommand{\eqr}[1]{\begin{eqnarray}\label{#1}}
\newcommand{\enr}{\end{eqnarray}}
\newcommand{\eqrst}[1]{\begin{eqnarray*}\label{#1}}
\newcommand{\enrst}{\end{eqnarray*}}

\makeatletter
\@addtoreset{equation}{chapter}
\makeatother

\numberwithin{equation}{section}

\newcommand{\ep}{\end{proposition}}
\newcommand{\bc}[1]{\begin{corollary}\label{#1}}
\newcommand{\ec}{\end{corollary}}
\newcommand{\bdf}[1]{\begin{definition}\label{\rm #1}}
\newcommand{\edf}{\end{definition}}
\newcommand{\bt}[1]{\begin{theorem}\label{#1}}
\newcommand{\et}{\end{theorem}}
\newcommand{\bl}[1]{\begin{lemma}\label{#1}}
\newcommand{\el}{\end{lemma}}
\newcommand{\bp}[1]{\begin{proposition}\label{#1}}
\newcommand{\br}[1]{\begin{remark}\label{#1}}
\newcommand{\er}{\end{remark}}
\newcommand{\bi}{\begin{description}}
\newcommand{\ei}{\end{description} }

  \newcommand{\beq}{\begin{equation}}
  \newcommand{\eeq}{\end{equation}}

\newcommand{\del}{\delta}

\newcommand{\RR}{\mathbb{R}}
\newcommand{\BR}{\mathbb{R}^{n}}

\newcommand{\lka}{\langle}
\newcommand{\rka}{\rangle}

  \newtheorem{theorem}{Theorem}[section]
\newtheorem{definition}[theorem]{Definition}
\newtheorem{remark}[theorem]{Remark}
  \newtheorem{lemma}[theorem]{Lemma}

\newtheorem{proposition}[theorem]{Proposition}

\newtheorem{corollary}[theorem]{Corollary}

\usepackage{textcomp}
\usepackage{eurosym}

\begin{document}
\title{Vorticity Gramian of compact Riemannian 
manifolds}
\author[1, 2 , $\ast ,$ 
\href{https://orcid.org/0000-0002-8628-0298}{iD1}]{Louis Omenyi}
\author[1 , 
\href{https://orcid.org/0000-0002-1051-5193}{iD2}]{Emmanuel Nwaeze}
\author[1,  
\href{https://orcid.org/0000-0003-1494-9850}{iD3}]{Friday Oyakhire}
\author[1, 
\href{https://orcid.org/0000-0003-2609-9785}{iD4}]{Monday Ekhator}

\affil[1]{Department  of  Mathematics and Statistics, 
 Alex Ekwueme Federal University, Ndufu-Alike, 
 Nigeria
($^{\ast}$Corresponding author E-mail: omenyi.louis@funai.edu.ng)\newline Corresponding author ORCID: 
0000-0002-8628-0298}
\affil[2]{Department  of  Mathematical Sciences, 
 Loughborough University, Loughborough, Leicestershire, 
 United Kingdom}

\date{\small Preprint submitted to RGN Publications 
on 06/November/2021}

\maketitle
\begin{abstract}
The vorticity of a vector field on $3$-dimensional Euclidean space is 
usually given by the curl of the vector field. In this paper, we extend 
this concept to $n$-dimensional compact and oriented Riemannian 
manifold. We analyse many properties of this operation. We prove that a 
vector field on a compact Riemannian manifold admits a unique Helmholtz 
decomposition and establish that every smooth vector field on an open 
neighbourhood of a compact Riemannian manifold admits a Stokes’ type 
identity.

\vskip1em \noindent \textbf{2010 AMS Classification:}
58J65; 58J30; 53C20  

\vskip1em \noindent \textbf{Article type:} Research Article  \ \
\end{abstract}

\section{Introduction}
Vorticity is a pseudovector field describing the local spinning 
motion of a flow near some point on a manifold. In classical mechanics,  
the dynamics of a flow are described by its rotation and expansion. 
Calin and Chang \cite{Cal} expressed the rotation component by the curl 
vector, while the expansion is described by the divergence function. 
The classical formulas involving rotation and expansion in the case 
of smooth functions and  vector fields on Riemannian manifolds show 
that gradient vector fields do not rotate and  that the curl vector 
field is incompressible. Varayu, Chew et al \cite{Var} among others 
showed that on Riemannian manifolds, the curl of a vector field is 
not a vector field, but a tensor. 

Many studies  on vorticity of flows on manifolds have been ongoing 
for many decades. For example, Frankel \cite{Fra} in fifties 
established how homology of manifolds influences vector and tensor 
fields on the manifolds. It showed that the vector fields give rise 
to one-parameter groups of point divergence-free transformations 
of the manifolds. The work of Xie and Mar \cite{Xie} employed Poisson 
equation for stream and vorticity equation to study $2$-dimensional 
vorticity and stream function expanded in general curvilinear 
coordinates. It constructed numerical algorithms of covariant, 
anti-covariant metric tensor and Christoffel symbols of the first 
and second kinds in curvilinear coordinates. Similarly, Perez-Garcia \cite{Per} studied 
exact solutions of the vorticity equation on the sphere as a manifold. 
In the work of  Peng and Yang \cite{Pen}, the existence of the curl 
operator on higher dimensional Euclidean space, $\BR , ~~ n > 3,$ 
was proved.

More recently, Bauer, Kolev and Preston \cite{Bau} carried out a 
geometric investigations of a 
vorticity model equation extending the works of  \cite{Vei, Can, Esc, 
Gus, Ign, Pak} and  Kim \cite{Kim} on  vorticity to manifold study.  
Besides, Deshmukh, Pesta and Turki \cite{Des} went ahead to show that 
the presence of a geodesic vector field on a Riemannian manifold 
influences its geometry while B$\ddot{a}$r \cite{Bar2} and 
M$\ddot{u}$ller \cite{Mul} extended the curl and divergence 
operators to odd-dimensional manifolds in arbitrary basis.

In this study, we extend the concept of vorticity to $n$-dimensional 
compact and oriented Riemannian manifolds and analyse many properties 
of this operation. We proceed with fixing our notations and briefly 
explaining some basic concepts required to follow the discussions. 

Let $(M,g)$ be a Riemannian manifold. By this we mean that $M$ is a 
topological space  that is locally similar to the Euclidean space and 
$g$ is the Riemannian metric on $M.$ We recall that a Riemannian metric 
$g$ on a smooth manifold $M$ is a symmetric, positive definite 
$(0,2)$-tensor field, see e.g. \cite{Jos, Lee, Agr} and \cite{Gri}. 
This means that for any  point $p \in M,$ the metric is the map 
$g_{p} : T_{p}M \times T_{p}M \rightarrow \RR $ that is a positive 
definite scalar product for a tangent space $T_{p}M.$ 
The Riemannian metric enables to measure distances, angles and 
lengths of vectors and curves on the manifold, see e.g. 
\cite{Bar, Bus, Cha, Jos} and \cite{Lee}, for details. We  denote the 
Riemannian manifold $(M,g)$ simply by $M.$  The manifold $M$ is called 
compact if it is compact as a topological space. If $M$ is a smooth 
manifold, then \cite{Cal} and \cite{Agr} proved that there is at 
least one Riemannian metric on $M.$ 

We call a function $f:M \rightarrow \RR$ smooth if for every chart 
$(U, \phi)$ on $M,$ and the function $f \circ \phi^{-1} : \phi (U) 
\rightarrow \RR$ is smooth. The set of all smooth functions on the 
manifold $M$ will be denoted by $C^{\infty}(M). $ 
Let $\Omega ^{k}$ denote the vector space of smooth 
$k$-forms on $M,$ and let $d: \Omega ^{k}\rightarrow \Omega ^{k+1}$ be 
the exterior derivative. Note that the metric which gives an inner 
product on the tangent space $T_{p}M$ at each $p \in M$ induces a 
natural metric on each cotangent space $T^{\ast}_{p}M,$ as follows. 
At $p,$ let $\{b_{1}, b_{2}, \cdots , b_{n} \}$ be an orthonormal basis 
for the tangent space. One obtains a metric on the cotangent space by 
declaring that the dual basis $\{b^{1}, b^{2}, \cdots , b^{n} \}$ is 
orthonormal. Hence given any two $k$-forms $\beta$ and $\gamma ,$ 
we have that $(\beta , \gamma )$ is a function on $M.$ We call 
$( \cdot , \cdot )$ the pointwise inner product; see e.g. \cite{Gri, 
Bus, Cha, Lee} and \cite{Ome}. For a coordinate chart on $M,$ 
\[(x^{1},\cdots ,x^{n}):U \rightarrow\BR ,\] 
we represent $g$ by the Gram matrix $(g_{ij})$ where  $g_{ij}= 
\lka \frac{\partial}{\partial x^{i}}, 
\frac{\partial}{\partial x^{j}} \rka ,$  
and $\lka , \rka$ is the inner product on the tangent space. 
The volume form $dV$ is defined be 
$b^{1} \wedge b^{2} \wedge \cdots \wedge b^{n},$ and it is a well-known 
fact from linear algebra that 
$dV= \sqrt{|g|}dx$  where $dx = dx^{1} \wedge \cdots \wedge dx^{n}.$ 
Using the point-wise inner product above, one writes  the $L^{2}$-inner 
product on $\Omega ^{k}(M)$  as 
\[\lka \beta , \gamma \rka = \int _{M} (\beta , \gamma) dV  ~~ 
\forall \beta , \gamma \in \Omega ^{k}(M). \]

\section{Vector fields and differential operators}
Vector fields and differential operators are the main tools used 
in the analysis of vorticity in this work. We employ these tools to 
construct the curl operator on $M$ and analyse its many properties. 

A vector field on $M$ is a family 
 $\{ X(p) \}_{p \in M}$ of tangent vectors such that $X(p) \in T_{p}M$
 for any $p\in M.$ In local coordinates chart 
 $(x^{1}, \cdots , x^{n}),$ 
 \[ X(p) = X^{i}(p)\frac{\partial}{\partial x^{i}|_{x=p}}. \]
The vector field $X(p)$ is called smooth if all functions $X^{i}$ are 
smooth in any chart in $M;$ see e.g. \cite{Agr, Bus} and \cite{Cal}.  
We denote the set of all vector fields on $M$  by $\Gamma (M).$ 

\begin{definition}[\cite{Cha, Gri}]
For every $p \in M$ the differential map $df$ at $p$ is defined by
\[ df_{p}:T_{p}M \rightarrow T_{f(p)}N ~~ \text{with}~~  
df_{p}(V)(h)= V(h \circ f) , ~~ \forall V \in T_{p}M,
 ~~ \forall h \in C^{\infty}(N). \]
Locally, it is given by 
\[df_{p}(\frac{\partial}{\partial x_{j|_{p}}})= 
\sum _{k=1}^{n}\frac{\partial f^{k}}{\partial x_{j|_{p}}}
\frac{\partial}{\partial y^{k}},\]
where $f = (f^{1}, f^{2}, \cdots , f^{n}).$  The matrix 
$\Big( \frac{\partial f^{k}}{\partial x_{j}} \Big)_{k,j}$ 
is the Jacobian of $f$ with respect to the charts 
$(x^{1},x^{2}, \cdots , x^{n})$ and $(y^{1},y^{2}, \cdots , y^{n})$ 
on $M$ and $N$ respectively. 
\end{definition}

\begin{definition}
Let $f\in C^{\infty}(M)$ be a smooth function. The gradient of $f,$ 
denoted by $\text{grad} f,$ is a vector field on $M$ metrically 
equivalent to the differential $df$ of $f:$
\begin{equation*}\label{lo1}
g(\text{grad} f , X) = df(X)= X(f ) , ~~ \forall X \in \Gamma(M). 
\end{equation*} 
\end{definition}

\begin{definition}
Let $X \in \Gamma(M)$  on $M.$ The  divergence of $X$ 
at the point $p \in M$ denoted as $\partial X$ is defined locally as 
\begin{equation*}\label{s2.10}
\partial X=\sum _{i=1}^{n}X^{i}_{;i} =
\sum _{i=1}^{n} \Big( \frac{\partial X^{i}}{\partial x_{i}} + 
\sum _{j}\Gamma _{ij}^{i}X^{j}  \Big); 
\end{equation*}
where 
\[ \Gamma_{jk}^{i} = \frac{1}{2}g^{il}
\Big( \frac{\partial g_{jl}}{\partial x_{k}} + 
 \frac{\partial g_{kl}}{\partial x_{j}} - 
 \frac{\partial g_{jk}}{\partial x_{l}}  \Big) \]
is the Christoffel symbol. In local coordinates, 
\[ \partial X = \frac{1}{\sqrt{|g|}}\frac{\partial}{\partial x_{j}}
(\sqrt{|g|}X^{j})\]
with summation over $j = 1, \cdots , n.$
\end{definition}
\begin{definition}
The Lie bracket 
$[,] : \Gamma(M) \times \Gamma(M) \rightarrow \Gamma(M)$ 
 is defined by
 \begin{equation*}\label{s10}
 [X,Y] = X(Y) - Y(X) ~~ \forall X,Y \in \Gamma(M).
 \end{equation*}
Locally, 
  \begin{equation*}\label{s11}
[X,Y] = \sum _{i,j=1}^{n}\Big( \frac{\partial Y^{i}}{\partial x_{j}}X^{j} 
 -  \frac{\partial X^{i}}{\partial x_{j}}Y^{j} \Big)
 \frac{\partial}{\partial x_{i}}.
 \end{equation*}
 If $[U, V] = 0,$ we say that the vector fields commute.
\end{definition}
 An extension of the usual directional derivative on the 
 Euclidean space to smooth manifold is linear connection.
\begin{definition}
A linear connection $\nabla$ on  $M$ is a map 
$\nabla : \Gamma (M) \times \Gamma (M) \rightarrow \Gamma (M)$ 
such that  $\nabla _{X}Y$ is $C^{\infty}(M)$ in $X$ and linear in $Y$ 
over the real field with 
 $\nabla _{X}(fY) = (Xf)Y + f\nabla _{X}Y, ~~ \forall f 
\in C^{\infty}(M).$ 
\end{definition}
We note that $\nabla _{X}Y$ is a new vector field which, roughly 
speaking, is the vector rate of change of $Y$ in the direction of $X.$ 
A particular connection on Riemannian manifolds that is torsion-free is 
the Levi-Civita connection. The Levi-Civita connection is defined in  
local coordinates as  
\[ \nabla _{X}Y = \sum _{i,k}^{n} X^{i}
\Big( \frac{\partial Y^{k}}{\partial x_{i}} + \sum _{j} 
\Gamma _{ij}^{k}W^{j}\Big) 
\frac{\partial}{\partial x_{k}}, \]
where 
\[ X = \sum _{i=1}^{n}X^{i}\frac{\partial}{\partial x_{i}}, ~~ 
Y = \sum _{k=1}^{n}Y^{k}\frac{\partial}{\partial x_{k}}  \]
and $\Gamma_{ij}^{k}$ are the Christoffel symbols defined by 
\[ \Gamma_{ij}^{k} = \frac{1}{2} \sum _{m}g^{km} 
\Big( \frac{\partial g_{jm}}{\partial x_{i}} + 
\frac{\partial g_{im}}{\partial x_{j}}
- \frac{\partial g_{ij}}{\partial x_{m}} \Big)  \]
where $(g^{km})$ is the inverse of $(g_{ij}).$
Note that $X$ is a Killing vector field if $\mathcal{L}_{X}g =0.$  
 
Let $M$ be a compact $n$-dimensional Riemannian manifold. 
A vector field on $M$ which generates isometries of the  Riemannian  
metrics is represented by a Killing vector field $v=v^{i}e_{i},$ where 
$e_{i}=\frac{\partial}{\partial x_{i}}.$  A Killing vector 
satisfies   the differential equation $v_{i}|_{j} = -  v_{j}|_{i}$ 
where solidus indicates covariant differentiation.  In particular, a 
Killing vector is divergence free with respect to the volume density 
$\sqrt{|g|},$ that is, 
\[ \frac{1}{\sqrt{|g|}}\frac{\partial}{\partial x_{i}}
\Big( \sqrt{|g|}v^{i}\Big) = v^{i}|_{i}= 0 \]
since a distance preserving map conserves volume automatically. Thus, 
a Killing vector, in particular, is an isometric flow. 
We also need to clarify the notion of tensor on $M.$
\begin{definition}
A tensor of type $(r, s)$ at $p \in M$ is a multi-linear function
\[ T: (T_{p}^{\ast}M)^{r} \times  (T_{p}M)^{s}  \rightarrow \RR .\]
A tensor field $\mathcal{T}$ of type $(r, s)$ is a smooth map, 
which assigns to each point $p \in  M$ an $(r, s)$-tensor 
$\mathcal{T}_{p}$ on 
$M$ at the point $p.$ In local coordinates,
\[ \mathcal{T} =  
\mathcal{T}_{j_{1}j_{2}\cdots j_{r}}^{i_{1}i_{2}\cdots i_{s}} 
dx^{j_{1}}\otimes dx^{j_{2}}\otimes  \cdots \otimes  dx^{j_{r}}\otimes
 \frac{\partial}{\partial x_{i_{1}}} \otimes \frac{\partial}{\partial 
 x_{i_{2}}} \otimes  \cdots \otimes 
 \frac{\partial}{\partial x_{i_{s}}} .\] 
 $\mathcal{T}$ acts on $r$ one-forms and $s$ vector fields  thus 
\begin{eqnarray*}
\mathcal{T}(\omega _{1}, \omega _{2}, \cdots , \omega _{r}, X_{1},
 X_{2} , \cdots , X_{s}) &=&  
\mathcal{T}_{j_{1}j_{2}\cdots j_{r}}^{i_{1}i_{2}\cdots i_{s}} 
dx_{j_{1}}(X_{1})dx_{j_{2}}(X_{2}) \cdots \\
&& dx_{j_{r}}(X_{r}) 
 \frac{\partial}{\partial x_{i_{1}}}(\omega _{1}) 
 \frac{\partial}{\partial x_{i_{2}}}(\omega _{2})  
 \cdots \frac{\partial}{\partial x_{i_{s}}}(\omega _{s}) \\
&=&
 \mathcal{T}_{j_{1}j_{2}\cdots j_{r}}^{i_{1}i_{2}\cdots i_{s}} 
X_{1}^{j_{1}} 
X_{2}^{j_{2}}
\cdots  
X_{r}^{j_{r}} 
\omega _{1}^{i_{1}}
\omega _{2}^{i_{2}}
\cdots 
\omega _{s}^{i_{s}} . 
\end{eqnarray*}
We say the tensor $\mathcal{T}$ is $s$ covariant and $r$ contravariant.
\end{definition}

Let $M$ be a compact $n$-dimensional manifold and let $\rho$ be a 
positive scalar density on $M.$ A $p$-tensor 
\[\omega = \omega ^{i_{1} \cdots i_{p}}e_{i_{1}}\wedge \cdots \wedge 
e_{i_{p}} \] 
is classically a skew-symmetric contravariant tensor of order $p;$ 
where $e_{i}=\frac{\partial}{\partial x_{i}}$ are the vectors of the 
coordinate frame (basis vectors). Let $T_{p}$ be the linear space 
of all $p$-tensors on all of the manifold, $M.$ We can now define 
the divergence of a tensor field. 
\begin{definition}
The divergence of a $p$-tensor $\omega$ written $\partial \omega$ is 
the $(p-1)$-tensor 
\begin{equation*}\label{divten}
\partial \omega = \frac{1}{\rho} 
\frac{\partial}{\partial x_{j}}\Big( \rho \omega ^{i_{1} \cdots i_{p}}
e_{i_{1}}\wedge \cdots \wedge e_{i_{p}} \Big). 
\end{equation*}
\end{definition}
Let  $V_{p}=\{ \omega \in T_{p}: \partial \omega =0 \}$
be the linear space of divergence-free $p$-tensors and let the linear 
space of $p$-tensor divergences be 
\[ D_{p}=\{ \omega \in T_{p}: \omega = \partial \omega ' ~~ 
\text{for some} ~~ \omega ' \in T_{p+1}   \} . \] 
For $p=0$ we have that the $p$-tensor is the ordinary scalar function. 
For such functions $f,$ we have that $\partial f =0.$ An easy 
calculation show that $\partial ^{2} - \partial \partial = 0,$ 
hence, $D_{p}$ is a linear subspace of $V_{p}.$ 
We have the following preliminary results in form of lemmas. 
\begin{lemma}\label{div2}
Let $X \in\Gamma(M),$ $T$ be an $(n,0)$-tensor field and $dV$ be the 
volume form  on $M.$ Let $\mathcal{L}_{X}$ be the Lie derivative of 
$T,$ then $\mathcal{L}_{X} dV = (\partial X )dV .$
\end{lemma}
\begin{proof}
We recall that $T = dV = \sqrt{|g|}dx_{1} \wedge 
dx_{2}\wedge \cdots \wedge dx_{n}$ 
is an $(n,0)$-tensor field on $M.$ The Lie derivative $\mathcal{L}_{X}$ 
of $T$ given by $\mathcal{L}_{X}(T) = T_{12 \cdots n} dx_{1} 
\wedge dx_{2}\wedge \cdots \wedge dx_{n}$ 
is also an $(n,0)$-tensor or an $n$-form
$\mathcal{L}_{X}(T) = (\mathcal{L}_{X}T)_{12 \cdots n} dx_{1} \wedge 
dx_{2}\wedge \cdots \wedge dx_{n} .$ 
We need to  show that
$(\mathcal{L}_{X}T)_{12 \cdots n} = (\partial X) \sqrt{g}.$ 
Indeed, using the formula which gives the components of the Lie 
derivative of a tensor, we have
\[(\mathcal{L}_{X}T)_{12 \cdots n} =  
\frac{\partial T_{12 \cdots n}}{\partial x_{i}}X^{i} + 
T^{j_{1}2 \cdots n} \frac{\partial X^{1}}{\partial x_{j_{1}}} + 
T^{2j_{2} \cdots n} \frac{\partial X^{2}}{\partial x_{j_{2}}}
+ \cdots +
T^{12 \cdots j_{n}} \frac{\partial X^{n}}{\partial x_{j_{n}}}. \]
As 
$T_{12 \cdots j_{p} \cdots n} = \delta _{p,j_{p}}T_{12 \cdots p
 \cdots n},$ 
we get 
\begin{eqnarray*}
(\mathcal{L}_{X}T)_{12 \cdots n} &=&  
\frac{\partial T_{12 \cdots n}}{\partial x_{i}}X^{i} + 
T_{12 \cdots n} (\frac{\partial X^{1}}{\partial x_{1}} + \cdots + 
\frac{\partial X^{n}}{\partial x_{n}} ) \\
&=& \frac{\partial T_{12 \cdots n}}{\partial x_{i}}X^{i} + 
T_{12 \cdots n} \frac{\partial X^{i}}{\partial x_{i}}\\
&=& \frac{\sqrt{g}}{\partial x_{i}} + 
\sqrt{g}\frac{\partial X^{i}}{\partial x_{i}}
= \frac{\partial }{\partial x_{_{i}}}\Big( \sqrt{g}X^{i} \Big)\\
&=& \frac{1}{\sqrt{g}} \frac{\partial }{\partial x_{i}}
 \Big( \sqrt{g} X^{i} \Big)\sqrt{g} = (\partial X \sqrt{g}.
\end{eqnarray*}
Hence, 
 $\mathcal{L}_{X} T = \partial X \sqrt{g} dx_{1} \wedge dx_{1} \wedge 
\cdots \wedge dx_{n} = \partial X dV .$ 
\end{proof}

\begin{lemma}
Let $f \in C^{\infty}(M)$ and $X \in \Gamma (M).$ Then
$\partial(fX) = f \partial X+ g(\text{grad} f, X).$ 
\end{lemma} 
\begin{proof}
From the definition of the $\partial$ operator, we have 
\begin{eqnarray*}
\partial(fX) &=& 
\frac{1}{\sqrt{g}} \frac{\partial }{\partial x_{j}}
 \Big( \sqrt{g} f X^{j} \Big)\sqrt{g} \\
 &=&  
 \frac{1}{\sqrt{g}} \frac{\partial f}{\partial x_{j}}\sqrt{g}X^{j} + 
 f \frac{1}{\sqrt{g}} \frac{\partial f}{\partial x_{j}}(\sqrt{g}X^{j}) \\
&=& \frac{\partial f}{\partial x_{j}} X^{j} +  f \partial X  \\
&=& g_{kj}(\text{grad} f)^{k}X^{j} + f \partial X \\
&=& g(\text{grad} f , X) + f \partial X.
\end{eqnarray*}
Using that $\mathcal{L}_{X}dV = (\partial X)dV~~ \forall X \in 
\Gamma(M),$ the result follows.
\end{proof}

We define the adjoint operator $\del :  \Omega ^{k} 
\rightarrow \Omega ^{k+1}$ of $d$ by requiring that
$\lka \omega , \del \beta \rka  = \lka d \omega , 
\beta \rka , ~ \forall 
\omega \in  \Omega ^{k-1}, ~~ \text{and} ~~ \beta \in \Omega ^{k}$ 
using the $L^{2}$ inner product  on $\Omega ^{k}(M).$ 
From the metric $g,$ we can also define the Hodge star operator 
$\ast : \Omega ^{k} \rightarrow  \Omega ^{n-k}$ 
by requiring that for all 
$\beta , \gamma \in  \Omega ^{k}, ~~ 
\beta \wedge \ast \gamma = (\beta , \gamma )dV.$ 
Notice that the Hodge star operator $\ast$ is linear and point-wise. 
Therefore, the inner product on $\Omega ^{k}$ can be written as 
$\lka \beta , \gamma \rka = \int _{M} \beta
 \wedge \ast \gamma . $ So, one can find an expression for $\del $  
 acting on one-forms. Given $f \in \Omega ^{0}$ and 
 $\omega \in \Omega ^{1},$ we 
have 
\begin{eqnarray*}
\lka f , \del \omega \rka &=& \lka df ,  \omega \rka = 
\int _{M} df \wedge \ast \omega \\
&=& \int _{M}[ d (f \wedge\ast \omega ) 
- f \wedge d \ast \omega ] \\
&=& \int _{M}-f \wedge d \ast \omega = -  \int _{M} f 
\wedge \ast ^{2} d \ast \omega \\
&=& \int _{M} f \wedge \ast (- \ast d \ast ) \omega = \lka f, (- \ast 
d \ast ) \omega \rka . 
\end{eqnarray*} 
Now we can go on to construct vorticity through the curl operator on 
the compact Riemannian manifold $M$ and study their flow vorticity.

\section{Vorticity of flows}
Let $X$ be a vector field on an open subset of $U \subset M$ and 
$\omega _{X}$ be the associated $1$-form  on $U$  dual to $X.$ So,  for 
each $p \in U,$ the linear functional $\omega _{X}(p)
\in T^{\ast}M$ on $T_{p}M$ is $\lka X(p) , \cdot \rka .$ 
It follows that the assignment $X \mapsto \omega _{X}$ is additive 
in $X$ and linear with respect to multiplication by smooth function 
on $U.$ Let $V$  be  a finite dimensional vector space with inner 
product $\lka , \rka ,$ then the dual vector space $V^{\ast}$ is 
naturally isomorphic  to $V$ under the map 
\[ \alpha : V \rightarrow V^{\ast} , ~~ \text{with}~~ 
\alpha (v) = v^{\ast} \in V^{\ast} , \text{satisfying} ~~ 
v^{\ast} (w) = \lka v , w \rka  ~ \forall v,w \in V .\]

These lead to define the curl operator taking vector fields to vector 
fields to be 
\[ \text{curl} X = \alpha ^{-1}\ast d \alpha X = g^{-1}\ast d g(X) .\]
In coordinate form, the curl of a vector field $X$ on a Riemannian 
manifold $M$ is a $2$-covariant antisymmetric tensor $A$ with the 
components $A_{ij}$  given by 
\[A_{ij}= X_{i;j}-X_{j;i} = \frac{\partial X_{i}}{\partial x_{j}} 
- \frac{\partial X_{j}}{\partial x_{i}}.\]
In particular, a Riemannian metric on a manifold $M$ is an 
assignment of inner product on each cotangent space 
$T^{\ast}_{p}M$ under the isomorphism $\alpha .$ The inner product 
$g$ induces an inner 
product on each of the tensor product $T_{p}M \otimes \cdots \otimes 
T_{p}M .$ 
\begin{definition}
Let $X \in \Gamma(U)$ with $U$ open in $M,$ the function 
$\text{div}(X) \in C^{\infty}(U)$  is characterised by 
\[ d(\ast (\omega _{X})) = \text{div}(X)dV_{M}\Big|_{U}.  \]
\end{definition}
To see this, let $M=\BR$ with the standard flat Riemannian metric and 
orientation in the standard linear coordinates 
$\{ x_{1}, \cdots , x_{n}\}.$ If $X=\sum X_{j}\partial _{x_{j}} 
\in \Gamma(U)$ is a vector field, since the volume form determined by 
this metric and orientation is $dx_{1}\wedge \cdots\wedge dx_{n},$ 
we have
\begin{eqnarray*}
d(\ast (\omega _{X})) &=& d(\ast (\sum X_{j}dX_{j})) 
= d(\sum X_{j} \ast (dX_{j})) \\
&=& d(\sum (-1)^{j-1}X_{j}dx_{1}\wedge \cdots \wedge 
\widehat{dx}_{n} \wedge  
\cdots  \wedge dx_{n})\\
&=& \sum (-1)^{j-1}dX_{j}\wedge dx_{1}\wedge \cdots \wedge 
\widehat{dx}_{n} \wedge  \cdots  \wedge dx_{n} \\
&=& \sum \Big(\frac{\partial X_{j}}{\partial x_{j}} \Big) 
dx_{1}\wedge \cdots   \wedge dx_{n} 
=  \sum \Big(\frac{\partial X_{j}}{\partial x_{j}}\Big)  dV_{M}
= \sum _{j=1}^{n}\frac{\partial X_{j}}{\partial x_{j}} 
= \text{div} (X) . 
\end{eqnarray*}

Our definition of divergence of a vector field was intrinsic 
to the Riemannian structure and orientation, thus, we can likewise 
compute the divergence in any oriented coordinate system on $M.$  We 
make the next definition. 
\begin{definition}
Let $M$ be a Riemannian manifold with corners. For an open 
$U \subseteq M$ and $f \in C^{\infty(U)},$ the smooth vector, 
$\text{grad}(f),$ on $U$ is defined by the condition:  
$\omega _{\text{grad}(f)} =df \in \Omega _{M}^{1}(U).$ 
That is, for each $p \in U,$ we have 
$\lka \text{grad}(f) , \cdot \rka = df(p)$  
is a linear functional on $T_{p}M.$
\end{definition}

Now let $M$ be a Riemannian manifold without boundary and $N$ be an 
oriented submanifold with boundary inside of $M$ with constant dimension 
$1.$ Let $N$ be given the induced Riemannian metric of $M.$ So the 
boundary $\partial N$ is assigned a collection of signs 
$\mathcal{E}(p) \in \{ \pm 1 \}$ for each $p \in \partial N$ where 
$\mathcal{E}$ is the usual Levi-Civita symbol defined by 
\[ \mathcal{E}_{a_{1}, \cdots ,a_{n}} = 
\begin{cases}
+ 1, & \text{if} ~~(a_{1}, \cdots, a_{n})~~\text{is an even permutation 
of }~~ 1,2, \cdots , n, \\
- 1, & \text{if} ~~ (a_{1}, \cdots, a_{n})~~\text{is an odd permutation 
of }~~ 1,2, \cdots , n ,\\
0, & \text{if otherwise} ;
\end{cases}\]
see e.g. \cite{Bar} and \cite{Bus}. Let $dl$ be the length form 
on $N$ and $T$ be the tangent field dual to $dl.$ It can be proved 
that for any $f \in C^{\infty}(M)$ for which $f\Big|_{N} 
\in C^{\infty}(N)$ is compactly supported, the smooth inner product 
function $\lka \text{grad}(f)\Big|_{N}, T \rka$ is compactly supported 
on $N$ and 
\[ \int _{N}  \lka \text{grad}(f)\Big|_{N}, T \rka dl = \sum _{p \in 
\partial N} \mathcal{E}(p) f(p). \]

In this way, we see that vector fields  give  rise to one-parameter 
groups of point transformations of the the manifold and one may be 
interested in those point transformations that are divergence-free.  
We call such vectors and their associated transformations simply 
``flows".  The next lemma ensures the existence of vector fields 
generating flows on $M.$
\begin{lemma}\label{recthm}
Let $V$ be a nonzero vector field at a point $p$ on
the manifold $M.$ Then there exists a system of coordinates 
$(\bar{x}^{1},\bar{x}^{2}, \cdots , \bar{x}^{n})$ about $p$ such
that there is $j \in \{1, \cdots , n\}$ for which 
$V = \frac{\partial}{\partial \bar{x}_{j}} .$  
We call this Rectification lemma. 
\end{lemma}
\begin{proof}
This lemma follows from the fact that in a compact Riemannian 
manifold, $M,$ if $p \in M$ there is an open neighbourhood $V$ of $p$ 
in the ambient manifold $\RR ^{n+1}$ such that if $U \subset \BR$ is 
open and $\phi : U \rightarrow \RR ^{n+1}$ is smooth, then 
$\phi (U)$ is a homomorphism. Besides, any transition Jacobian on $U$ 
for change of coordinates has full rank for every $p \in U.$  This proves the lemma.
\end{proof}

Given a vector field $X,$ consider the system
\begin{equation}\label{s6}
\frac{d c^{k}}{dt} = X^{k}(c(t)), ~~ k= 1,2, \cdots , n,
\end{equation}
where $c(t)$ is the integral curve associated with $X.$ 
The next result shows that the system \eqref{s6} can be solved locally 
around the point $x_{0}=c(0), $ for $0<t<\epsilon .$ 

\begin{proposition}[Existence and uniqueness]
Given $x_{0} \in M$ and let $X$ be a nonzero
vector field on an open set $U \subset M$ of $x_{0},$ 
then there is 	 $\epsilon > 0$ such that the system \eqref{s6}
has a unique solution $c:[0,\epsilon ) \rightarrow U$ such that 
$c(0) = x_{0}.$ 
\end{proposition}
\begin{proof}
By the rectification lemma, there is a local change 
of coordinates $\bar{x} = \phi (x)$ such that the system \eqref{s6} 
becomes
\[\frac{d c^{k}}{dt} = \delta _{kn}, ~~ k= 1,2, \cdots , n  ~~
\text{where} ~~ \bar{c}=\phi (c).\]
This system has a unique solution through the point 
$\bar{x}_{0}=\phi (x_{0})$ given by  
\[\bar{c}^{k}(t) = \bar{x}^{k}_{0},
~~ k= 1,2, \cdots , n-1 ~~ \text{and} ~~ 
\bar{c}^{n}(t) = t + \bar{x}^{n}_{0}.\]
Hence this will hold also for the system \eqref{s6} 
in a neighbourhood of $x_{0}=\phi ^{-1}(\bar{x}_{0}).$
\end{proof}

Let $f_{j}:=\frac{\partial f}{\partial x_{j}},$ and $f^{i}:=
g^{ij}f_{j}$ so that $\nabla f 
= f^{i}\frac{\partial }{\partial x} 
=g^{ij}\frac{\partial f}{\partial x_{j}}\frac{\partial }{\partial x} ,$ 
we have more propositions that will lead us to the main results of 
this work.
\begin{proposition}
If $X \in \Gamma (M)$ then $X = \text{grad}\phi \iff \text{curl}X =0$ 
and $tr (\text{curl}X)=0.$
\end{proposition}
\begin{proof}
Let $X = \text{grad}\phi ,$ then 
$X^{k}=g^{kj}\frac{\partial \phi}{\partial x_{j}}$ and $X_{i}=
\frac{\partial \phi}{\partial x_{i}}.$ So we have 
\[ (\text{curl}X)_{ij} =  
\frac{\partial X_{i}}{\partial x_{j}} 
- \frac{\partial X_{j}}{\partial x_{i}} = 
\frac{\partial ^{2}\phi}{\partial x_{j}\partial x_{i}} - 
\frac{\partial ^{2}\phi}{\partial x_{i}\partial x_{j}}  =0. \]
 
Conversely, let $X \in \Gamma(M)$ such that $\text{curl}X =0.$  Then, 
$\frac{\partial X_{i}}{\partial x_{j}} = 
\frac{\partial X_{j}}{\partial x_{i}}.$ Thus the $1$-form 
$\omega = \sum X_{k}dx_{k}$ is exact. This means there is a locally 
defined function such that 
$\omega = df = \sum  \frac{\partial f}{\partial x_{k}} dx_{k}.$  
Thus, $X_{k}=X^{j}.$  Besides, $tr (\text{curl}X) = g^{ij}(X_{i;j}-
X_{j;i}) = X_{;j}^{j}- X_{;i}^{i} = 0 ;$ which completes the proof. 
\end{proof}

Let $\lka , \rka$ be the Riemannian metric with associated  Levi-Civita 
connection $\nabla .$ We prove global and 
invariant properties of the curl operator on $M.$ 
\begin{proposition}
If $A = \text{curl} X,$ then  
$ A(U,V) = \lka \nabla _{V}X , U \rka  - 
\lka \nabla _{U}X , V \rka ~~ \forall U,V \in \Gamma (M). $
\end{proposition}
\begin{proof}
For every $U, V \in \Gamma(M) ,$  we have 
\begin{eqnarray*}
A(U,V) &=& A_{ij}U^{i}V^{j}=(X_{i;j} - X_{j;i})U^{i}V^{j}
= (\nabla _{\partial _{j}}X)_{i}U^{i}V^{j} - 
(\nabla _{\partial _{i}}X)_{j}U^{i}V^{j} \\
&=&
\lka \nabla _{\partial _{j}}X , U\rka V^{j} - 
\lka \nabla _{\partial _{i}}X , U\rka U^{i}
= \lka \nabla _{V^{j}\partial _{j}}X , U\rka - 
\lka \nabla _{U^{i}\partial _{j}}X , V\rka \\
&=& \lka \nabla _{V}X , U \rka  - 
\lka \nabla _{U}X , V \rka .
\end{eqnarray*}
\end{proof}

\begin{proposition}\label{s8.6L}
 Let $A = \text{curl} X,$ where $X \in \Gamma(M).$ 
 Then  
\[A(U,V) = V\lka X , U \rka  - U \lka X , V \rka 
+ \lka X , [U , V] \rka . \]
\end{proposition}
\begin{proof}
Since $\nabla$ is a metric connection
\[V \lka X , U \rka   = \lka \nabla _{V}X , U \rka + 
\lka X, \nabla _{V} U \rka ~~ \text{and} ~~  
U\lka X , V \rka   = \lka \nabla _{U}X , V \rka + 
\lka X, \nabla _{U} V \rka .\]
Using the symmetry of $\nabla,$ subtracting we obtain
\[V \lka X , U \rka  - U \lka X , V \rka  = A(U,V) 
+ \lka X, [V,U]\rka ,\] 
which  proves the claim.
\end{proof}
 
The following result shows the relation between the curl, 
Levi-Civita connection and the Lie derivative.
\begin{proposition}\label{thmC}
If $A = \text{curl} X$ and $\nabla$ is the Levi-Civita connection on 
$M,$ then 
\[A(U,V) = 2 \lka \nabla _{V}X, U\rka - (L_{X}g)(U,V). \]
\end{proposition}
\begin{proof}
We have
\[ 2 \lka \nabla _{V}X, U\rka = 
V \lka X, U\rka   + X \lka U, V\rka - U \lka V, X \rka
- \lka V, [X,U]\rka  + \lka X, [U,V] \rka
+ \lka U, [V, X]\rka .  \] 
That is, 
\begin{eqnarray*}
 2 \lka \nabla _{V}X, U\rka &=&  A(U,V) +  X \lka U, V\rka 
 - \lka V, [X,U]\rka + \lka U, [V, X]\rka  \\
 &=& A(U,V) +  X \lka U, V\rka 
 - \lka V, L_{X} U \rka - \lka U, L_{X} V \rka .
\end{eqnarray*}
Using  that $(L_{X}g)(U,V) = X \lka U, V\rka 
 - \lka  L_{X} U, V \rka - \lka U, L_{X} V \rka ,$ 
we obtain  
\[2 \lka \nabla _{V}X, U\rka =  A(U,V) +   (L_{X}g)(U,V).\]
\end{proof}

We can now show a Helmholtz decomposition of vector fields on $M.$ 
This is the theorem that follows. That is, we prove that a vector 
field $X$ on a compact Riemannian manifold can be uniquely 
decomposed as a sum of two vectors $Y$ and $Z,$ where $Y$ is the 
rotation component and $Z$ the expansion component. 
\begin{theorem}\label{mainhoz}
If $X \in \Gamma(M),$ there are two vector fields $Y$ and  $Z$ on $M$ 
such that $X = Y + Z,$ with $\text{div}Y = 0$ and $\text{curl}Z = 0.$ 
Moreover, the decomposition is unique.
\end{theorem}
\begin{proof}
Let $\eta = \text{div}X$ and let $\phi$ solve the elliptic equation 
$\nabla \circ \nabla \phi =  \Delta \phi = \eta .$ 
Take $Z = \nabla \phi$ and $Y = X - \nabla \phi .$ 
Then $\text{curl}Z = \text{curl}\nabla \phi =0$ and 
$\text{div}Y = \eta - \nabla \phi =0.$ This proves the existence of 
$Y$ and $Z.$

Now suppose two decompositions of $X$ so that 
$X = Y_{1} + Z_{1}= Y_{2} + Z_{2} .$ 
As $\text{curl}Z_{1} =0,$ it follows that there are two functions 
$\phi_{i},~i=1,2$ such that $Z_{i}=\nabla\phi_{i},~i=1,2.$ \\
So, $Y_{2}-Y_{1} = \nabla (\phi_{2} - \phi_{1}). $
 
Denote $W=Y_{2}-Y_{1}$ and $\phi = \phi_{2} - \phi_{1} ,$ then 
$\text{div}W = \text{div} \nabla \phi .$ 
Since $\text{div}Y_{2}-\text{div}Y_{1}=0$ we get 
$\Delta \phi =0,$ thus,  $\phi_{2} - \phi_{1}$ must be 
constant. Taking the gradient yields $Z_{2}-Z_{1} =0.$ 
Then we have also  that $Y_{1}=Y_{2},$ hence the decomposition is 
unique. 
\end{proof}
For example, let $X=(x_{1}-x_{2})\partial _{x_{1}} + 
(x_{1}+x_{2})\partial _{x_{2}}.$
Then the Helmholtz decomposition is 
\[X=Y+Z ~~ \text{with} ~~  
Y =x_{1}\partial _{x_{1}}  + x_{2}\partial _{x_{2}} ~~ \text{and} ~~  
Z= - x_{2}\partial _{x_{1}}  + x_{1}\partial _{x_{2}} .\]
To visualise this, we consider the $2$-dimensional unit sphere with a 
local parametrisation 
\[ \Phi : (-\frac{\pi}{2} , \frac{\pi}{2}) 
\rightarrow \RR ^{3} ~~ \text{with}~~ \Phi (\theta , \phi )=
\begin{pmatrix}
\cos \theta \cos \phi \\
\cos \theta \sin \phi \\
\sin \theta 
\end{pmatrix} . \]
The differentials $\frac{\partial \Phi}{\partial \theta}$ and 
$\frac{\partial \Phi}{\partial \phi}$ give the Gramian 
$\displaystyle{g = \Big( g_{ij} \Big)_{ij} =  
\begin{pmatrix}
1 & 0 \\
0 & \cos ^{2}\theta
\end{pmatrix}}$
which has the associated vector field flow 
$X_{S^{2}} = (1,\cos ^{2} \theta )$ and stream plot as 
Figure  \ref{plotX2}.   
Observe that $X_{S^{2}}$ is divergence-free and that 
$\text{curl}X_{S^{2}} = - 2 \cos \theta \sin \theta .$
\begin{figure}[hbtp]
\centering
\includegraphics[scale=0.5]{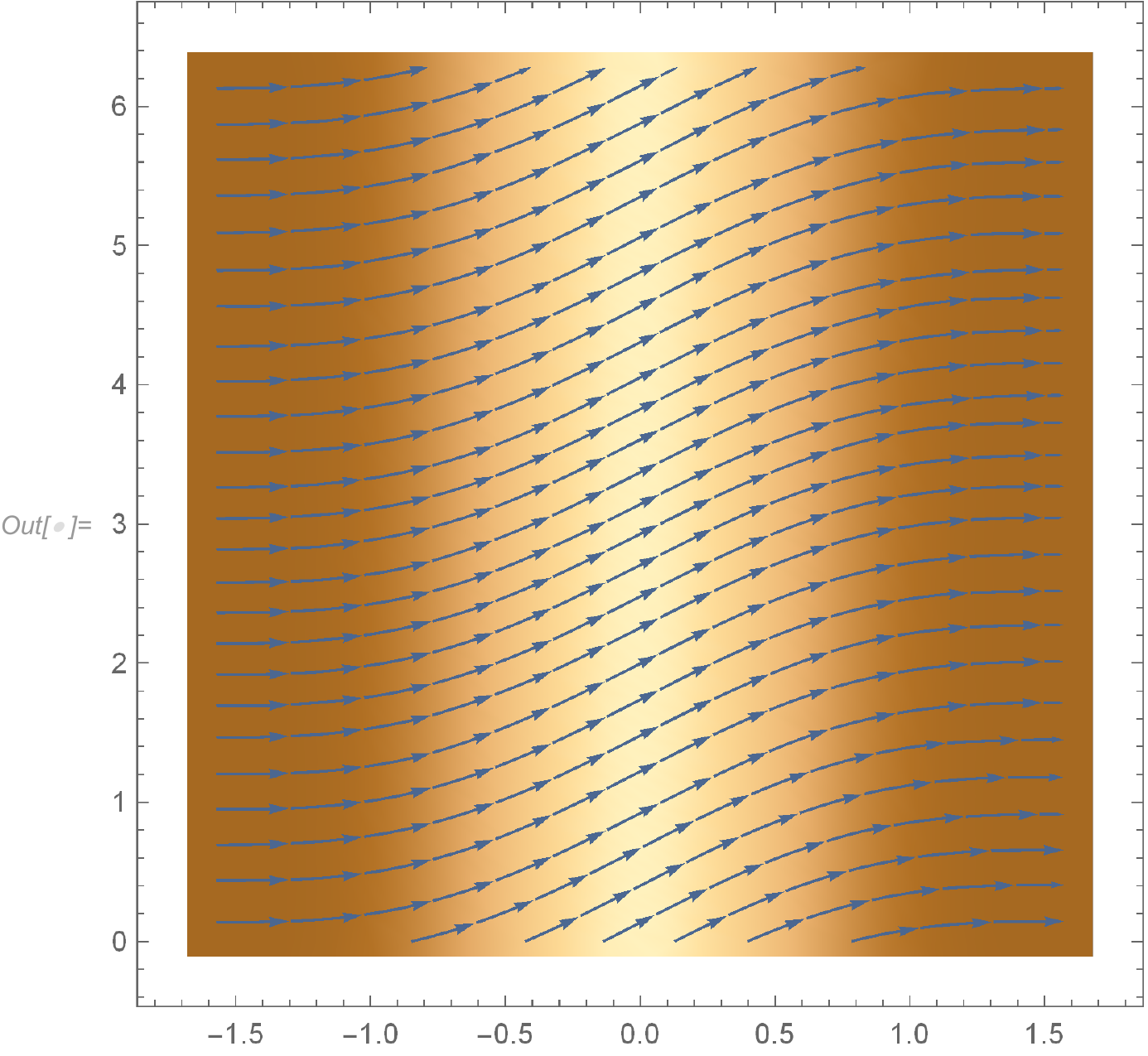}
\caption{\small{Stream plot of  the vector field 
$X_{S^{2}}$ on $S^{2}.$}}
\label{plotX2}
\end{figure}
By the Helmholtz decomposition, 
\[X_{S^{2}} = Y_{S^{2}} + Z_{S^{2}} ~~ \text{with} ~~Y_{S^{2}} = X_{S^{2}} ~~ \text{and} ~~ Z_{S^{2}} = (- \cos ^{2} \theta , 1)~~ 
\text{with}~~ \text{curl}Z_{S^{2}} = 0.\]
The stream plot of the Helmholtz decomposed $X_{S^{2}}$ flow is 
Figure \ref{plotYZ2}. 
\begin{figure}[hbtp]
\centering
\includegraphics[scale=0.5]{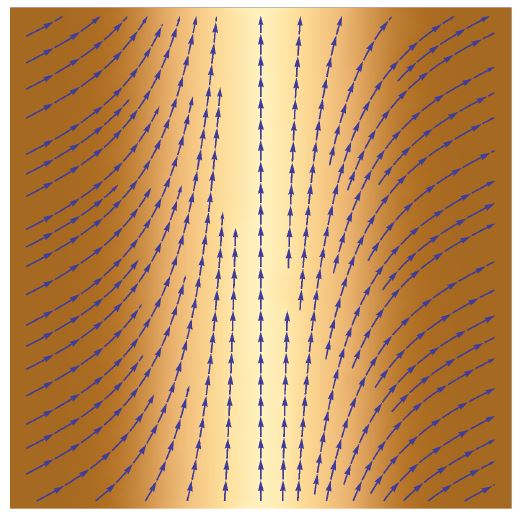}
\caption{\small{Stream plot of the decomposed $X_{S^{2}}.$}}
\label{plotYZ2}
\end{figure}\\

Finally, another main result of this work is the Stokes' type result:
\begin{theorem}
Let $X$ be smooth vector field on an open neighbourhood $U$ of 
$Z \subseteq  M$ with $X\Big|_{Z}$  compactly supported and let 
$dA$ be the area form on $Z.$ Then the smooth function 
$\lka \text{curl}(X)\Big|_{Z}, \hat{N} \rka$ 
on $Z,$ where $\hat{N}$ is the outward unit normal field along $Z$ 
in $M,$ is compactly supported and satisfies the Stokes'
 type identity 
\begin{equation}\label{smain}
\int _{Z}  \lka \text{curl}(X)\Big|_{Z}, \hat{N} \rka dA = 
\int _{\partial Z}  \lka X\Big|_{\partial Z}, T \rka dl.
\end{equation}
\end{theorem}
\begin{proof}
Let $\eta = \omega_{X} \in \Omega_{M}^{1}(U)$ be dual to $X.$ Thus 
$\eta \Big|_{Z} \in \Omega_{Z}^{1}(Z)$ is compactly supported as it 
vanishes at points where $X\Big|_{Z}$ vanishes. By the definition of 
the curl, we have $\omega _{\text{curl}(X)} = \ast (d \eta).$ 
Since $\ast \circ \ast = (-1)^{r(n-r)}$ on $r$-forms on an $n$-
dimensional manifold, for $n=3$ and $r=1,$ we have 
$\ast \circ \ast = 1$ on $1$-forms on $U,$ so $d \eta 
= \ast (\omega _{\text{curl}(X)}).$ 
Similarly, 
\[\lka \text{curl}(X)\Big|_{Z}, \hat{N} \rka dA = (d \eta )\Big|_{Z} \] 
inside of $\Omega_{Z}^{2}(Z).$ But 
$(d \eta)\Big|_{Z}=d(\eta \Big|_{Z})$ since $d$ and pullback commute 
along the closed embedding of $Z$ into $M.$ This is compactly supported 
on $Z.$ For compactly supported $1$-form $\eta \Big|_{Z}$ on the $2$-
dimensional manifold $Z$ with boundary, we get 
\[ \int _{Z}  \lka \text{curl}(X)\Big|_{Z}, \hat{N} \rka dA = 
\int _{ Z}  (d \eta) \Big|_{Z} = 
\int _{ Z}  d( \eta \Big|_{Z}) = 
\int _{\partial Z}  \eta \Big|_{\partial Z} \]
showing \eqref{smain}. 

Thus our problem comes to proving the identity  
$\eta \Big|_{\partial Z}=  \lka X\Big|_{\partial Z}, T \rka dl
\in \Omega _{\partial Z}^{1}(\partial Z).$ 
But since $\eta |_{Z} = \omega _{X|_{Z}} \in 
\Omega _{ Z}^{1}(Z)$ it follows that 
$\omega _{X|_{Z}} = \omega _{X} Z$ pointwise.

Let $X|_{Z}=G$ and $G|_{\partial Z}=H ,$ we want to prove that for any 
smooth vector field $H$ along $\partial Z ,$
\begin{equation}\label{id1}
\lka H, T \rka dl = \omega_{H}
\in \Omega _{\partial Z}^{1}(\partial Z)
\end{equation}
 generally for any $1$-dimensional Riemannian manifold $C$ with length 
 form $dl$ dual to the tangent field $T.$  Evaluating both sides of 
 \eqref{id1} at a point $p \in C,$ we obtain 
a $1$-dimensional real vector space $V$ endowed with an inner product. 
So, let $V=T_{p}C,$ if $v \in V$ with length for $\phi ,$ and for 
$t \in V$ the vector dual to $\phi$  we have $\lka v ,t \rka \phi = 
\lka v , \cdot \rka$ is in the dual space. It suffices to check this 
equality when evaluating both sides on the basis $\{ t \}.$ But 
since $\phi (t) =1$ by the definition of $t,$ the result follows. 
\end{proof}

\section{Conclusion}
We have constructed vorticity of vector field flows on compact smooth 
Riemannian manifolds through differential operators mainly the curl 
and the divergence operators. Many properties of vorticity on the 
manifolds were established and we proved that a vector field on a 
compact Riemannian manifold admits  a unique Helmholtz decomposition. 
We also proved  a Stokes' type identity for the curl operator on smooth 
tensor fields on $M.$

This study can be extended to applications of the central ideas of 
the paper to physical flow problems in engineering and industry. 
One can as well study the vorticity of climate variabilities on the 
earth surface using the machinery developed in this paper. This will 
throw more light on the mathematical analysis of climate change. 

\section*{Acknowledgement}
The authors are grateful to the members of Research and Development 
Committee of the Department of Mathematics and Statistics of the Alex 
Ekwueme Federal University, Ndufu-Alike, Nigeria for their comments and 
suggestions that have improved the readability of this paper. We also 
thank members of our families for their cooperation while this research 
lasted. All errors are our responsibility.


\begin{thebibliography}{99}
\bibitem{Agr} I.~Agricla and T.~Friedrich, Global Analysis - 
Differential Forms in Analysis, Geometry and Physics, 
Graduate Studies in Mathematics Volume 52, 
\emph{American Mathematical Society,} (2002), 
\url{https://bookstore.ams.org/gsm-52}.

\bibitem{Ati} M.~F.~Atiya, V.~K.~Patodi and I.~M.~Singer, 
Spectral asymmetry and Riemannian geometry I,  
\emph{Math. Proc. Cambridge Philos. Soc.} \textbf{77} (1975), 43 -- 69, 
\doi{10.1017/S0305004100049410}.

\bibitem{Bar} C.~B$\ddot{a}$r, Elementary Differential 
Geometry, \emph{Cambridge University Press,} United Kingdom, (2010),  
\doi{10.1017/CBO9780511844843}.

\bibitem{Bar2} C.~B$\ddot{a}$r,     
The curl operator on odd-dimensional manifolds,  
\emph{Journal of Mathematical Physics} \textbf{60}, (031501), (2019),  
\doi{10.1063/1.5082528}.

\bibitem{Bau} M.~Bauer, B.~Kolev, S.~Preston,   
Geometric investigations of a vorticity model equation,  
\emph{Journal of Differential Equations} \textbf{260}, (2016), 478 -- 516, 
 \doi{10.1016/j.jde.2015.09.030}.

\bibitem{Bus} P.~Buser, Geometry and Spectra of compact Riemann Surfaces, 
\emph{Birkh$\ddot{a}$user,}  (1992), 
\url{https://link.springer.com/book/10.1007/978-0-8176-4992-0}.

\bibitem{Cal} O.~Calin and D.~C.~Chang, Geometric mechanics on 
Riemannian manifolds: Applications to partial differential equations, 
\emph{Birkh$\ddot{a}$uster,} Boston, Basel, Berlin, (2005), 
\url{https://link.springer.com/book/10.1007/b138771}.

\bibitem{Can}  J.~Cantarella, D.~DeTurch, H.~Gluck and 
M.~Teytel, The spectrum of the curl operator on 
spherically symmetric domains, \emph{Physics of Plasmas 7} \textbf{2766},  (2000), 
\doi{10.1063/1.874127}.

\bibitem{Cha}  I.~Chavel, Eigenvalues in  Riemannian 
Geometry, \emph{Academic Press Inc,} London, (1984), 
\doi{10.1016/s0079-8169(08)x6051-9}.

\bibitem{Des}  S.~Deshmukh , P.~Peska and N.~B.~Turki, 
Geodesic vector fields on a Riemannian manifold,  
\emph{Mathematics MDPI} \textbf{8}, 137, (2020), \doi{10.3390/math8010137}.

\bibitem{Esc} J.~Escher, The geometry of a vorticity model equation, 
\emph{Communications on Pure and Applied Analysis} \textbf{11}, 4, (2012), 1407 -- 1419, 
\doi{10.3934/cpaa.2012.11.1407}.

\bibitem{Fra}  T.~Frankel, Homology and flows on manifolds,  
\emph{Annals of Mathematics} \textbf{65}, 2, (1957), 330 -- 339, 
\url{https://www.jstor.org/stable/1969965}.

\bibitem{Gri}  A.~Grigor'yam,  Heat Kernels and Analysis 
on Manifolds, Studies in  Advance Mathematics 
Volume 47, \emph{American Mathematical Society,} (2009), 
\doi{10.1090/amsip/047}.

\bibitem{Gus} B.~Gustafsson, 
Vortex motion and geometric function theory: the role of connections, 
\emph{Philosophical Transactions A} \textbf{377}:20180341,  (2019), 
\doi{10.1098/rsta.2018.0341}.
 
\bibitem{Ign} R.~Ignat and R.~L.~Jerrard,  
Interaction energy between vortices of vector fields on  Riemannian surfaces, 
\emph{Comptes Rendus Mathematique,} (2017), 
 \doi{10.1016/j.crma.2017.04.004}.

\bibitem{Jos} J.~Jost, Riemannian Geometry and geometric analysis, $5^{th}$   edition,  
\emph{Springer-Verlag}, Berlin Heidelberg, Germany, (2008), 
\doi{10.1007/978-3-540-77341-2}.

\bibitem{Kim} S.~C.~Kim, Vortex motion on Riemannian surfaces, 
\emph{Journal of Korean Physical Society} 
\textbf{59}, 1, (2011),  47 -- 54, 
\doi{10.3938/jkps.59.47}.

\bibitem{Lee}  J.~M.~Lee, Introduction to smooth manifolds,   
Graduate Texts in Mathematics, \emph{Springer,}  New York, (2003), 
\url{https://link.springer.com/book/10.1007/978-1-4419-9982-5}.

\bibitem{Mul}  D.~M$\ddot{u}$ller, W.~P.~F.~de Medeiros,  
R.~R.~de~Lima, V.~C.~de~Andrade,  The divergence and curl in arbitrary basis, 
\emph{Revita Brasileira de Ensino de Fisica} 
\textbf{41}, 2, (2019), 1 -- 7, 
\doi{10.1590/1806-9126-RBEF-2018-0082}.

\bibitem{Ome} L.~Omenyi and U.~Uchenna, Global Analysis on 
Riemannian manifolds, \emph{The Australian  Journal of Mathematical 
Analysis and Applications}  \textbf{16}, 02, 11, (2019), 1 -- 17, 
\url{https://ajmaa.org/searchroot/files/pdf/v16n2/v16i2p11.pdf}.

\bibitem{Pak} H.~C.~Pak, Motion of vortex filaments in 3-
manifolds, \emph{Bull. Korean Math. Soc.} \textbf{42}, 1, (2005),  75 -- 85, 
\doi{10.4134/BKMS.2005.42.1.075}.

\bibitem{Pen} L.~Peng and L.~Yang, The curl in seven dimensional space and its application, 
\emph{Approx. Theory \& its 
Appl.} \textbf{15}, (1999), 66 -- 80, 
\doi{10.1007/BF02837124}.

\bibitem{Per}  I.~Perez-Garcia, Exact solutions of the 
vorticity equation on the sphere as a manifold, \emph{Atmosfera} 
\textbf{28}, 3, (2015), 179 -- 190, 
\doi{10.20937/ATM.2015.28.03.03}.
  
\bibitem{Var}  B.~Varayu, T~S.~Chew, P.~Y.~Lee,  
On the divergence theorem on manifolds, \emph{J. Math. Appl.} 
\textbf{397}, (2013), 182 -- 190, 
\doi{10.1016/j.jmaa.2012.07.042}. 
  
\bibitem{Vei} H.~B.~da Veiga and L.~C.~Berselli, 
Navier-Stokes equation: Green matrices, vorticity direction and regularity up to the boundary,  
\emph{Journal of Differential Equations} \textbf{246}, (2009), 597 -- 628, 
\doi{10.1016/j.jde.2008.02.043}.
  
\bibitem{Xie} X.~L.~Xie and W.~W.~Mar, 2D Vorticity and stream 
function solutions based on curvilinear coordinates,  
\emph{Third International Conference on Experimental Mechanics, 
Proceedings of SPIE} \textbf{4537},  (2002), 
\doi{10.1117/12.468772}.


\end{thebibliography}
\end{document}